\documentclass{amsart}[11pt]

\usepackage{amssymb,latexsym,amsmath,amsthm,amscd,verbatim,mathrsfs}

\usepackage[left=3cm,top=3cm,right=3cm,bottom = 3cm]{geometry}

\usepackage{tikz-cd}
\tikzset{>=latex}

\usepackage{hyperref}
\usepackage{xcolor}
\hypersetup{
    colorlinks,
    linkcolor={red!50!black},
    citecolor={blue!50!black},
    urlcolor={blue!80!black}
}

\newtheoremstyle{def}
     {10pt}
     {10pt}
     {}
     {}
     {\rmfamily\bfseries\upshape}
     {.}
     {.5em}
     {}
 \theoremstyle{def}
 \newtheorem{definition}{Definition}[subsection]
\newtheorem{remark}[definition]{Remark}

\newtheoremstyle{thm}
     {20pt}
     {10pt}
     {\it}
     {}
     {\rmfamily\bfseries\upshape}
     {.}
     {.5em}
     {}

\theoremstyle{thm}

\newtheorem{theorem}[definition]{Theorem}
\newtheorem{lemma}[definition]{Lemma}
\newtheorem{proposition}[definition]{Proposition}

\DeclareMathOperator{\End}{End}
\DeclareMathOperator{\Aut}{Aut}

\DeclareMathOperator{\Spec}{Spec}

\DeclareMathOperator{\Hom}{Hom}

\DeclareMathOperator{\id}{id}

\DeclareMathOperator{\pr}{pr}
\DeclareMathOperator{\Pic}{Pic}

\newcommand{\smalltwobytwo}[4]{
\bigl( \begin{smallmatrix} 
  #1 & #2\\
  #3 & #4 
\end{smallmatrix} \bigr)}

\newcommand{\cL}{\mathcal{L}}
\newcommand{\cM}{\mathcal{M}}

\newcommand{\cO}{\mathcal{O}}

\newcommand{\sA}{\mathscr{A}}

\newcommand{\sM}{\mathscr{M}}

\newcommand{\CC}{\mathbb{C}}

\newcommand{\GG}{\mathbb{G}}

\newcommand{\ZZ}{\mathbb{Z}}

\usepackage{pgf, tikz}
\usetikzlibrary{arrows, automata}

\makeatletter
\def\@seccntformat#1{%
  \protect\textup{\protect\@secnumfont
    \ifnum\pdfstrcmp{subsection}{#1}=0 \bfseries\fi
    \csname the#1\endcsname
    \protect\@secnumpunct
  }%
}  
\makeatother

\usepackage{tikz-cd}
\tikzset{>=latex}

\begin{document}

\title{The algebraic functional equation of Riemann's theta function}

\author{Luca Candelori}
\date{}
\maketitle

\begin{abstract}
We give an algebraic analog of the functional equation of Riemann's theta function. More precisely, we define a `theta multiplier' line bundle over the moduli stack of principally polarized abelian schemes with theta characteristic and prove that its dual is isomorphic to the determinant bundle over the moduli stack. We do so by explicitly computing with Picard groups over the moduli stack. This is all done over the ring $R=\ZZ[1/2,i]$: passing to the complex numbers, we recover the classical functional equation. 
\end{abstract}

\section{Introduction}

In 1829 C.G.J Jacobi introduced the theta function $\vartheta(\tau) = \sum_{n\in \ZZ} e^{\pi i n^2 \tau}$ and proved the remarkable transformation formula
\begin{equation}
\label{equation:FEqDim1}
\vartheta\left(\frac{-1}{\tau}\right) = \sqrt{\frac{\tau}{i}}\,\vartheta(\tau), \quad \tau \in \mathfrak{h} := \{ z\in \CC: \mathrm{Im}[z]>0\},
\end{equation} 
where $\sqrt{\cdot}$ is defined using the principal branch of the logarithm. Along with the trivial identity $\vartheta(\tau + 2) = \vartheta(\tau)$, equation \eqref{equation:FEqDim1} expresses the fact that $\vartheta$ is a {\em modular form} of weight 1/2 on the {\em theta group} $\Gamma_1^+(1,2)\subseteq \mathrm{SL}_2(\ZZ)$, a level 2 congruence subgroup. Much effort has been devoted in recent years to study the algebro-geometric nature of modular forms, which has led to great progress in understanding their Fourier coefficients and their relation to Galois representations. Yet despite all this success, statements such as \eqref{equation:FEqDim1} have remained rather mysterious from an algebro-geometric point of view, their algebraic nature obscured by the use of analytic techniques such as Poisson summation. Given the importance of theta functions and of modular forms of half-integral weight in general, this has to be viewed as a serious gap in our understanding of the theory. It is the aim of this paper to fill this gap. 

The natural geometric framework to study this problem is the moduli space of principally polarized abelian varieties of dimension $g$ with theta characteristic. Over the complex numbers this moduli space is given by the analytic quotient stack $\sA^{\pm}_{g,\rm{an}}:=\left[\Gamma^{\pm}_g(1,2)\backslash \mathfrak{h}_g\right]$, where $\mathfrak{h}_g$ is the $g$-dimensional Siegel upper half-space and, according to whether the characteristic is even (+) or odd (-), the subgroups $\Gamma^{\pm}_g(1,2) \subseteq \mathrm{Sp}_{2g}(\ZZ)$ are the {\em theta groups}, generalizing $\Gamma_1^+(1,2)$ above (e.g. \cite{MoretBailly:PinceauxAbv}, VIII, \S 3.4.1, \S 3.4.2). For example, in the even case we have 
\begin{equation}
\label{equation:thetaGroup}
\Gamma^+_g(1,2) := \left\{ \smalltwobytwo{A}{B}{C}{D} \in \mathrm{Sp}_{2g}(\ZZ): AB^t \text{ and } CD^t \text{ have even diagonal } \right\},
\end{equation} 
which is the group of transformations of the {\em Riemann theta function} $\vartheta_g(\tau)$. This is the analytic function of $\tau \in \mathfrak{h}_g$ given by 
$$
\vartheta_g(\tau):= \sum_{n\in \ZZ^g} e^{\pi i\, n^t\tau n }, 
$$
a higher-dimensional generalization of $\vartheta(\tau)$. In this geometric setting, the theta function $\vartheta_g(\tau)$ is the value at zero of a section of a certain relatively ample, normalized, symmetric line bundle $\Theta$ of degree 1 over the universal abelian variety $\pi:\sA \rightarrow \Gamma^{+}_g(1,2)\backslash \mathfrak{h}_g$, whose isomorphism class is determined by the theta characteristic. In particular, $\vartheta_g$ is a section of the dual $(\pi_*\Theta)^{-1}$ of the line bundle $\pi_*\Theta$. A similar theta function and corresponding geometric analog can be obtained in the case of odd theta characteristics.

Theta functions such as $\vartheta_g(\tau)$ have well-known functional equations analogous to the $g=1$ case \eqref{equation:FEqDim1}. In this article, we would like to explain the relationship between these transformation laws and the moduli problem at hand. Moreover, we want to do this {\em algebraically}: namely, we work over the algebraic moduli stack $\sA^{\pm}_g$ of principally polarized abelian schemes with theta characteristic. This is an algebraic stack which is smooth over $\ZZ[1/2]$ (\cite{MoretBailly:PinceauxAbv}, \S VIII.3.2.4), and whose analytification is the analytic stack $\sA^{\pm}_{g,\rm{an}}$ defined above. The formation of $\vartheta_g$ can be carried out in this algebraic setting as well. However, any analytic statement such as \eqref{equation:FEqDim1} seems now meaningless, since  $\vartheta_g$ is no longer an analytic function on $\mathfrak{h}_g$, but rather a section of an algebraic line bundle over an algebraic stack. Our goal then is to translate the functional equation into a statement which makes sense over any commutative ring, and not just over the complex numbers.   

The first results in this direction have been obtained in \cite{MoretBailly:PinceauxAbv}, where the Grothendieck-Riemann-Roch Theorem is employed to show that the {\em determinant bundle}
$$
\Delta(\Theta):= \pi_*\Theta^{\otimes 2}\otimes\underline{\omega},
$$ 
where $\underline{\omega}$ is the Hodge bundle of $\pi$, is a canonical torsion element in $\Pic(\sA^{\pm}_g)$ (this result is the `canonical key formula' of \cite{MoretBailly:PinceauxAbv}, App. 1). This fact alone shows that some high power $r$ of the algebraic Riemann theta function $\vartheta_g$ is an algebraic Siegel modular form of weight $r/2$, i.e. a section of $\underline{\omega}^{r/2}$. This result was later improved by \cite{FaltingsChai}, Theorem I.5.1, who show that the smallest such $r$ is 8 and essentially that $\Delta(\Theta)$ has the structure of a $\mu_4$-torsor over $\sA^{\pm}_g$. Moreover, it was shown in \cite{MB:FEqn} that the trivialization of $\Delta(\Theta)^{\otimes 4}$ can be computed analytically using the functional equation of Riemann's theta function. In fact, the functional equation \eqref{equation:FEqDim1} and its higher-dimensional analogs, together with the above-mentioned results, suggest that there must exist a natural `theta multiplier bundle' $\cM(\Theta)$, a $\mu_4$-torsor over $\sA^{\pm}_g$, and an isomorphism
\begin{equation}
\label{equation:keyFormulaIntro}
\Delta(\Theta)\stackrel{\simeq}\longrightarrow  \cM(\Theta)^{-1}
\end{equation}
of the underlying line bundles over $\sA^{\pm}_g$. Over $\CC$, $\cM(\Theta)$ must correspond to the character $$\lambda:\Gamma(1,2)^{\pm}_g\rightarrow \mu_4 \subseteq \CC^{\times}$$giving the 4-th roots of unity appearing in the transformation laws of $\vartheta_g^2(\tau)$. We construct the theta multiplier bundle $\cM(\Theta)$ in Section \ref{section:thetaMultiplierBundle} and prove the isomorphism \eqref{equation:keyFormulaIntro} in Section \ref{section:canonicalKeyFormula}. Moreover, we show in Section \ref{section:functionalEquation} that \eqref{equation:keyFormulaIntro} is given analytically by the functional equation. The current article can thus be viewed as a refinement of \cite{FaltingsChai}, Theorem I.5.1 and \cite{MB:FEqn}, and it was very much inspired by these works.   

Our construction of theta multiplier bundles is based on a geometric interpretation of $\lambda$, which could be of independent interest. This interpretation is due to Pierre Deligne (private communication), and is given in Section \ref{section:thetaMultipliers}. In the even case, a different geometric interpretation of $\lambda$ had been already given in the beautiful paper \cite{JM}: that construction was also one of the original inspirations for this article and  the relation with our $\lambda$ is given at the end of Section \ref{section:thetaMultipliers}. 

There are many other interesting modular forms on levels other than $\Gamma^{\pm}_g(1,2)$, whose algebraic construction can be similarly worked out. In fact, the isomorphism \eqref{equation:keyFormulaIntro} can be generalized to the case when $\Theta$ is replaced by a non-degenerate line bundle of higher degree and index, in which case \eqref{equation:keyFormulaIntro} is replaced by an isomorphism of vector bundles. This level of  generality requires the machinery of Heisenberg groups and will be presented in follow-up work. This future work will give algebraic definitions of modular form (of integral and half-integral weight) taking values in Weil representations: in particular, it will be possible to define algebraic modular forms of half-integral weight in the classical sense, that of Shimura. Those notions may then be employed, for example, as the starting point of a theory of mod $p$ and $p$-adic analogs of such very interesting types of modular forms, or to initiate an algebraic study of their Fourier coefficients.        

I would like to acknowledge the great help that I have received from Pierre Deligne while drafting this paper. The ideas of Sections \ref{section:thetaMultipliers} and \ref{section:thetaMultiplierBundle} are entirely due to him, along with many corrections he suggested from earlier drafts. In addition, this work was initiated while I was a graduate student of Henri Darmon at McGill University. I sincerely thank him for introducing me to this problem and for his generous support over the years. Also at McGill, I have enjoyed many fruitful conversations with and feedback from Eyal Goren, Niky Kamran and Cl\'{e}ment Gomez. At LSU, I must thank William J. Hoffman, Ling Long and Karl Mahlburg for their comments and suggestions. At Princeton/IAS, I must also thank Peter Sarnak and Shou-Wu Zhang for their feedback during my visit there. I would also like to thank Cameron Franc for his comments, and James Parson for his corrections from an earlier draft. Finally, I would like to acknowledge the Schulich and Pellettier families for their financial support for this research, along with the McGill and LSU Departments of Mathematics.

\section{Symplectic 4-groups with theta characteristic}
\label{section:thetaMultipliers}

\begin{definition}
A {\em symplectic 4-group} is a free $\ZZ/4\ZZ$-module $V$ of rank $2g$ together with a non-degenerate alternating bilinear form
$$
\psi: V\times V \longrightarrow \ZZ/4\ZZ.
$$
\end{definition}

Let $\overline{V}:=V/2V$ be the free $\ZZ/2\ZZ$-module of rank $2g$ obtained from $V$ by reducing the coordinates of $V$ modulo 2. The bilinear form $2\,\psi$ descends to the quotient $\overline{V}$ to give a  non-degenerate alternating (thus symmetric) bilinear form  $\overline{\psi}$ on $\overline{V}$. Suppose then we are given a function
$$
q: \overline{V}\longrightarrow \ZZ/2\ZZ \\
$$
which is a quadratic form for $\overline{\psi}$, i.e. 
$$
q(v_1+v_2)- q(v_1) - q(v_2) = \overline{\psi}(v_1,v_2), \quad\text{for all}\; v_1,v_2\in\overline{V}.
$$
\begin{definition}
The triple $(V,\psi,q)$ is called a {\em symplectic 4-group with theta characteristic}. The group $\Gamma:= \Aut(V,\psi,q)$, of $\ZZ/4\ZZ$-linear automorphisms $\varphi:V\rightarrow V$ such that $\varphi$ preserves $\psi$ and $\overline{\varphi}$ preserves $q$, is called the {\em theta group} of $(V,\psi,q)$ (cp. \cite{JM}). 
\end{definition}

Up to isomorphism, there are only two quadratic forms $q$ for each $g$. We say that $q$ is {\em even} if there exists a subspace $L\subseteq \overline{V}$ such that $q(L) = 0$ and $\dim L = g$ (i.e. $L$ is {\em maximal isotropic}), and {\em odd} otherwise. We say that the symplectic 4-group $(V,\psi)$ with theta characteristic $q$ is {\em even} or {\em odd}, according to the parity of $q$. In terms of the Arf invariant, $q$ is even if $\mathrm{Arf}(q) = 0$ and odd otherwise. Accordingly, the isomorphism class of the theta group of $(V,\psi,q)$ is entirely determined by $g$ and by the parity of $q$. There are special elements of this group which are essential to understand its structure.

\begin{definition}[\cite{JM}, \S 1]
A {\em anisotropic transvection} is any linear map $t_v \in \Gamma$ of the form
$$
t_v(x) = x + \psi(v,x)\,v,
$$
where $v$ is a vector such that $q(\overline{v}) \neq 0$. 
\end{definition}  

Note that any such anisotropic transvection $t_v$ satisfies
\begin{itemize}
\item[(i)] $t^2_v(x) = x + 2\,\psi(v,x)\,v$,
\item[(ii)] $\gamma t_v\gamma^{-1} = t_{\gamma v}$, $\;\gamma \in \Gamma$.
\end{itemize}
In particular, (i) implies that $t^2_v$ reduces to the identity modulo 2 and (ii) implies that all anisotropic transvections are conjugates of each other.

\subsection{} Given a symplectic 4-group with theta characteristic $(V,\psi,q)$, its  theta group $\Gamma$ is a group extension
\begin{equation}
\label{equation:exSeq}
0\rightarrow \Gamma(2) \rightarrow \Gamma \stackrel{\gamma\mapsto \bar{\gamma}}\longrightarrow \mathrm{O}(\overline{V},q) \rightarrow 0.
\end{equation} 
Each flanking term in this exact sequence is endowed with a natural homomorphism to a cyclic group of order 2, which we now describe. First, we have the {\em Dickson invariant} (e.g. \cite{Dieudonne}, \S 3)
$$
D_q: \mathrm{O}(\overline{V},q) \longrightarrow \ZZ/2\ZZ,
$$
defined by the action of $\mathrm{O}(\overline{V},q)$ on $Z(\mathrm{Cliff}^{+}(\overline{V},q))$. Given an isomorphism $\mathrm{O}(\overline{V},q)\simeq \mathrm{O}^{\pm}(2g,2)$, there are explicit formulas for $D_q(t)$ which are quadratic  in the entries of the matrix $t$. Alternatively, $D_q$ is given by the formula (\cite{Wilson}, \S 3.8.1)
\begin{equation}
\label{equation:DicksonFormula}
D_q(t) = \mathrm{rk}(\mathrm{id} + t) \mod 2.
\end{equation}  

Second, there is a homomorphism
$$
q: \Gamma(2) \longrightarrow \ZZ/2\ZZ,
$$
canonically induced by $q$, defined as follows. There is an isomorphism
\begin{align*}
\Gamma(2) &\stackrel{\simeq}\longrightarrow \mathfrak{sp}(\overline{V},\overline{\psi}) \\
\alpha = \id + 2\beta &\longmapsto \beta,
\end{align*}
as $\ZZ/2\ZZ$-vector spaces of rank $g(2g+1)$, and the symmetric bilinear form $\overline{\psi}$ induces a canonical isomorphism $$\mathfrak{gl}(\overline{V},\overline{\psi}) \simeq \End(\overline{V})\simeq \overline{V}^*\otimes\overline{V} \stackrel{\overline{\psi}}\simeq \overline{V}\otimes\overline{V},$$ under which $\mathfrak{sp}(\overline{V},\overline{\psi}) $ corresponds to the subspace of symmetric 2-tensors $\Sigma^2(\overline{V})$. Therefore $\Gamma(2)$ and  $\Sigma^2(\overline{V})$ are canonically isomorphic as $\ZZ/2\ZZ$-vector spaces. The quadratic form $q$ now induces a linear form
$
q: \Sigma^2(\overline{V}) \longrightarrow \ZZ/2\ZZ,
$
by the universal property of $\Sigma^2(\overline{V})$ with respect to degree 2 maps.

It turns out that the two homomorphisms $D_q$ and $q$ can be combined to construct a remarkable $\ZZ/4\ZZ$-valued character on $\Gamma:=\Aut(V,\psi,q)$:

\begin{theorem}
\label{theorem:thetaCharacter}
Let $(V,\psi,q)$ be a symplectic 4-group with theta characteristic, and let $\Gamma=\Aut(V,\psi,q)$ be its theta group. Then there is a unique group homomorphism 
$$
\lambda:\Gamma \longrightarrow \ZZ/4\ZZ,
$$
such that
\begin{itemize}
\item[(i)] $\lambda|_{\Gamma(2)} = 2\cdot q$, where $\ZZ/2\ZZ\stackrel{2\cdot}\rightarrow \ZZ/4\ZZ$ is the canonical injection, 
\item[(ii)] $\lambda(\gamma) \equiv D_q(\overline{\gamma}) \mod 2$, for all  $\gamma\in \Gamma$,
\item[(iii)] $\lambda(t_v) = 1$, for any anistropic transvection $t_v$.  
\end{itemize}

\end{theorem}

\begin{proof}

The orthogonal group $\mathrm{O}(\overline{V},q)$, by definition, preserves the quadratic form $q$ and therefore it preserves  $\ker(q:\Gamma(2)\rightarrow \ZZ/2\ZZ)$ under the outer action given by \eqref{equation:exSeq}. The group $\ker(q)$ is thus normal in $\Gamma$, and the quotient is a central extension
\begin{equation}
\label{equation:centralExt1}
0\rightarrow \Gamma(2)/\ker(q) \rightarrow \Gamma/\ker(q) \stackrel{p}\rightarrow \mathrm{O}(\overline{V},q)\rightarrow 0.
\end{equation}
Indeed, $\Gamma(2)/\ker(q) = \{ \pm \id + \ker(q) \}$ since $q(-\id) = \det(\overline{\psi})$ and $\overline{\psi}$ is non-degenerate.

Next, consider $\Omega(\overline{V},q) := \ker(D_q)$, the `special orthogonal group' in characteristic 2. The central extension
$$
0\rightarrow \Gamma(2)/\ker(q) \rightarrow p^{-1}(\Omega(\overline{V},q)) \stackrel{p}\rightarrow \Omega(\overline{V},q)\rightarrow 0,
$$
deduced from \eqref{equation:centralExt1}, has a {\em unique} splitting $\sigma$ for all $g\geq 5$. In fact, for such $g$ the group $\Omega(\overline{V},q)$ is simple with trivial Schur multiplier (\cite{Wilson}, \S 3.8.2). We may thus form the quotient $ G:=\left(\Gamma/\ker(q)\right)/\sigma(\Omega(\overline{V},q))$, a central extension of $\Gamma(2)/\ker(q) \stackrel{q}\simeq \ZZ/2\ZZ$ by $\mathrm{O}(\overline{V},q)/\Omega(\overline{V},q) \stackrel{D_q}\simeq \ZZ/2\ZZ$. 

We now claim that there is an isomorphism of $G$ with $\ZZ/4\ZZ$. To find this isomorphism, let $t_v$ be any anisotropic transvection. Clearly $D_q(\overline{t}_v) = 1$ from \eqref{equation:DicksonFormula}, so $t_v \neq 0 \in G$. Moreover, $t^2_v\in \Gamma(2)$ so that
$$
\lambda(t_v^2) = q(t_v^2) = q(\bar{v}\otimes \bar{v}) = q(\bar{v}) \neq 0,$$ 
thus $t_v$ gives an element of exact order 4 in $G$. 

To summarize, for $g\geq 5$ we have a commutative diagram

$$
\begin{tikzcd}[column sep=0.1in,row sep=0.25in]
0 \arrow{r} & \Gamma(2) \arrow{r} \arrow{d} & \Gamma \arrow{r} \arrow{d} & \mathrm{O}(\overline{V},q) \arrow{r} \arrow{d}{\mathrm{id}}& 0 \\
0 \arrow{r} & \Gamma(2)/\ker(q)\arrow{r}\arrow{d}{\mathrm{id}} &\Gamma/\ker(q) \arrow{r} \arrow{d} & \mathrm{O}(\overline{V},q) \arrow{r} \arrow{d}& 0 \\
0 \arrow{r} & \Gamma(2)/\ker(q)\arrow{r}\arrow{d}{q} & G \arrow{r}\arrow{d}{\simeq\; t_v}  & \mathrm{O}(\overline{V},q)/\Omega(\overline{V},q) \arrow{r}\arrow{d}{D_q} & 0 \\
0 \arrow{r} & \ZZ/2\ZZ \arrow{r}{2\cdot}&\ZZ/4\ZZ \arrow{r}{\mod 2}  & \ZZ/2\ZZ \arrow{r} & 0. \\
\end{tikzcd}
$$ 
We define $\lambda:\Gamma\rightarrow \ZZ/4\ZZ$ to be the homomorphism given by composing the arrows in the middle vertical column. Properties (i) and (ii) are then clear from the definition. Property (iii), and uniqueness, follow by requiring one (all) anisotropic transvection $t_v$ to map to 1 under the isomorphism $G\simeq \ZZ/4\ZZ$. If $g< 5$, we may choose an embedding $(V,\psi,q)\hookrightarrow ( V\oplus V', \psi\oplus \psi',q\oplus q')$ into a symplectic 4-group $V\oplus V'$  with large enough $g$, in which case there is a canonical injection 
$$
\Aut(V,\psi,q) \hookrightarrow \Aut( V\oplus V', \psi\oplus \psi',q\oplus q'),
$$
and we may define $\lambda$ on $\Gamma = \Aut(V,\psi,q)$ by restriction. 
\end{proof}

By construction, $\lambda$ satisfies the following important compatibility:

\begin{proposition}
\label{proposition:directSums}
For any two symplectic 4-groups with theta characteristic $(V,\psi,q),(V',\psi',q')$, the diagram
$$
\begin{tikzcd}
\Gamma(V)\times\Gamma(V') \arrow[hook]{r} \arrow{d}{\lambda(V) + \lambda(V')}  & \Gamma(V\oplus V')\arrow{d}{\lambda(V\oplus V')} \\
\ZZ/4\ZZ \arrow{r}{=} & \ZZ/4\ZZ
\end{tikzcd}  
$$
is commutative. 
\end{proposition}

\subsection{Case $g=1$, $q$ even}
\label{section:lambdaGenus1Even} By choosing a basis for $V$ we may assume that $V\simeq \ZZ/4\ZZ^{\oplus 2}$, $\psi$ is the standard symplectic form 
$$
\psi(x_1,y_1,x_2,y_2) = x_1y_2 - y_1x_2,
$$
and $q(x,y) = xy$. In this case $\mathrm{O}(\overline{V},q) \simeq \mathrm{O}^+(2,2)$, 
$$
\mathrm{O}^+(2,2) = \left\{ \smalltwobytwo{1}{0}{0}{1}, \smalltwobytwo{0}{1}{1}{0} \right\}\simeq \ZZ/2\ZZ,
$$
and $\Gamma(2)\simeq \ZZ/2\ZZ^{\oplus 3}$. The theta group $\Gamma$ is generated by $S=\smalltwobytwo{0}{-1}{1}{0}$ and $T^2=\smalltwobytwo{1}{2}{0}{1}$. Now $T^2\in \Gamma(2)$ and it corresponds to the symmetric 2-tensor $\bar{v_1}\otimes \bar{v_1}$, where $v_1 = (1,0)$. Thus $\lambda(T^2) = q(\bar{v_1}\otimes \bar{v_1}) = q(1,0)=0$. On the other hand $D_q(\overline{S}) = 1$, thus $\lambda(S) = \pm 1$. The sign can be fixed by choosing an anisotropic transvection. For example, the vector $(1,1) \in V$ reduces mod 2 to the unique anisotropic vector of $\overline{V}$, and the corresponding transvection $t_v$ is given by the matrix $\smalltwobytwo{0}{1}{-1}{2}$. Setting $\lambda(t_v) = 1$ then forces $\lambda(S) = -1$, since $S\,t_v = T^2$. 

\subsection{Case $g=1$, $q$ odd}
\label{section:lambdaGenus1Odd}
 Again let $V = \ZZ/4\ZZ^{\oplus 2}$, $\psi$ as above and $q$ now given by
$
q(x,y) = x^2 + y^2 + xy.
$
We have $\mathrm{O}(\overline{V},q) \simeq \mathrm{O}^-(2,2)\simeq S_3$ and $\Gamma(2)\simeq \ZZ/2\ZZ^{\oplus 3}$. The theta group $\Gamma$ is isomorphic to $\mathrm{SL}_2(\ZZ/4\ZZ)$ and is generated by $S$ and $T=\smalltwobytwo{1}{1}{0}{1}$. Now $T = t_{v_1}$ is the anisotropic transvection of the vector $v_1=(1,0)$, thus $\lambda(T)=1$. On the other hand $\lambda(S) = \pm 1$ as before. To fix the sign, let $v=(1,1)$ as before and apply $\lambda$ to $S\,t_v = T^2$ to obtain $\lambda(S) = 1$. 

\subsection{Case $g\geq 3$, $q$ even} By Proposition \ref{proposition:directSums}, in order to compute $\lambda$ explicitly it suffices to find a formula for it in the case of $q$ even and $g$ large (say $g \geq 3$). In this case, a different construction of $\lambda$ has already been given in \cite{JM}, together with a remarkably simple algorithm to compute it. We now recall this construction and show how it is a special case of ours. 

\begin{definition}
Let $(V,\psi,q)$ be a symplectic 4-group of rank $2g$ with even theta characteristic. A free $\ZZ/4\ZZ$-submodule $L\subseteq V$ is called an {\em isotropic lagrangian} if it is a direct summand of $V$ of rank $g$, with $\psi = 0$ on $L\times L$ $q(\overline{L}) = 0$.  
\end{definition}

Let $L$ be any isotropic lagrangian, and let $\{v_i\}_{i=1}^g$ be a $\ZZ/4\ZZ$-basis for it. Any two such bases differ by a uniquely defined element of $\mathrm{GL}(L)$, of determinant $\pm 1$. We may thus define an equivalence relation on the set of all bases for $L$ by declaring $\{v_i\}_{i=1}^g \sim \{v'_i\}_{i=1}^g$ if they differ by an element of determinant 1. There are only 2 such equivalence classes, which we call {\em orientations} of $L$.  

\begin{definition}
A pair $(L,[\{v_i\}_{i=1}^g])$ of an isotropic lagrangian $L$ and a choice of orientation $[\{v_i\}^g_{i=1}]$ is called an {\em oriented} isotropic lagrangian. The set of all such pairs is denoted by $\Lambda_0(V)$. 
\end{definition}

The authors of \cite{JM} define the function $m_{JM}: \Lambda_0(V)\times\Lambda_0(V) \rightarrow \ZZ/4\ZZ$ by
$$
(L_1,L_2) \longmapsto \sigma(L_1,L_2) + (g-\dim {\bar{L}_1\cap\bar{L}_2}) - 1 \mod 4,
$$
where $\sigma(L_1,L_2) \in \{\pm 1\}$ is a sign function, depending on the orientations, defined as follows: if $L_1\cap L_2 = \{0\}$ (i.e. $L_1$ and $L_2$ are {\em transversal}), then $\psi: L_1\rightarrow L^*_2$ is an isomorphism, and $\sigma$ is the determinant of the matrix of this isomorphism with respect to the orientations given. If $\overline{L}_1 = \overline{L}_2$, then we choose an isotropic lagrangian $D$ which is transversal to both $L_1$ and $L_2$ and set $\sigma(L_1,L_2) = \sigma(L_1,D)\sigma(L_2,D)$. All other cases can be reduced to these two (\cite{JM}, \S 2). To define a character $\Gamma\rightarrow \ZZ/4\ZZ$, fix an oriented isotropic lagrangian $L_0$ and let
$$
\lambda_{JM}(\gamma):= m_{JM}(L_0, \gamma L_0). 
$$
This is a homomorphism satisfying $\lambda_{JM}(t_v) =1$ (\cite{JM}, \S 3) for any anisotropic transvection $t_v$, thus $\lambda_{JM} = \lambda$ since the abelianization of $\Gamma$ is equal to $\ZZ/4\ZZ$, generated by the conjugacy classes of anisotropic transvections (\cite{JM}, Theorem 1.1.(i)). 

Conversely, the function $m_{JM}$ is entirely determined by the character $\lambda_{JM}$, since this is trivial on commutators. In particular, $m_{JM}$ can be recovered from our definition of $\lambda$. To see this, note that the theta group $\Gamma = \Aut(V,\psi,q)$ acts transitively on $\Lambda_0(V)$, and we may define a function
\begin{align*}
m: \Lambda_0(V)\times \Lambda_0(V) &\longrightarrow \ZZ/4\ZZ \\
(L_1,L_2) &\longmapsto \lambda(\gamma_{1,2}),
\end{align*}
by choosing $\gamma_{1,2} \in \Gamma$ so that $\gamma_{1,2}L_1 = L_2$.

\begin{lemma} 
The function $m$ is well-defined, and $m=m_{JM}$.  
\end{lemma}

\begin{proof}
Choose an oriented isotropic lagrangian $L$ and a splitting $V = L\oplus M$ compatible with $\psi$ and $q$. Let $\Gamma(L)$ be the stabilizer of $L$ under the action of $\Gamma$ on $\Lambda_0(V)$,  i.e. the subgroup of all $\gamma \in \Gamma$ which preserve $L$, along with its orientation. We need to show that $\lambda$ factors through $\Gamma(L)$. Now any element of $\Gamma(L)$, with respect to the chosen orientation, has the form $\smalltwobytwo{A}{B}{0}{A^{-1,t}}$, where $A\in \mathrm{SL}(L)$. In other words, $\Gamma(L)$ is a group extension
$$
0 \rightarrow   T \longrightarrow \Gamma(L) \stackrel{\gamma\mapsto A}\longrightarrow \mathrm{SL}(L) \rightarrow 0,
$$
where elements of $T$ are of the form $\smalltwobytwo{I}{B}{0}{I}$. Since $\mathrm{SL}(L)$ has no non-trivial characters, it suffices to show that $\lambda$ factors through $T$. Indeed, $T$ is a group extension
$$
0 \rightarrow   T(2) \longrightarrow T \stackrel{\mod 2}\longrightarrow \overline{T} \rightarrow 0,
$$
where elements of $T(2)\subseteq \Gamma(2)$ are of the form $\smalltwobytwo{I}{B}{0}{I} = I +2\smalltwobytwo{0}{B'}{0}{0}$. These elements map under $\bar{\psi}$ to symmetric 2-tensors in $\Sigma^2(L)$, where the value of $\lambda = q$ is always 0 since $L$ is isotropic. Thus $\lambda$ factors through $T(2)$. But $\lambda(\gamma) \equiv D_q(\bar{\gamma}) \mod 2$, and clearly the elements of $T$ have even rank mod 2, thus $\lambda$ factors through all of $T$. Now the value of $m$ is entirely determined by $\lambda = \lambda_{JM}$, and the same is true for $m_{JM}$.   
\end{proof}

\begin{remark}
For $q$ odd, note that $\Lambda_0(V) = \emptyset$ and the geometric construction of \cite{JM} cannot be applied directly.  
\end{remark}

\section{Theta multiplier bundles} 
\label{section:thetaMultiplierBundle}Let now $S$ be a scheme over $\ZZ[1/2,i]$. where $i$ denotes a choice of primitive 4th root of unity. Let $K\rightarrow S$ be a finite \'{e}tale commutative group scheme with geometric fibers isomorphic to
$
\mathbb{Z}/4\mathbb{Z}^{\oplus 2g},
$
together with a non-degenerate symplectic pairing
$$
e_K: K\times K \longrightarrow \mu_4,
$$
where $\mu_N$ is the finite flat $S$-group scheme of $N$-th roots of unity. We call $(K,e_K)$ a {\em symplectic 4-group scheme} of rank $2g$. The group scheme $\overline{K}: = K/2K$ is endowed with the alternating (thus symmetric) pairing $e_{\overline{K}}:= e^2_{K}$. Let 
$$
e^K_*: \overline{K}\longrightarrow \mu_2
$$
be a quadratic character for $e_{\overline{K}}$. 

\begin{definition}
The triple $(K,e_K,e_*^K)$ is called a {\em symplectic $4$-group scheme with theta characteristic}. The {\em theta group} of $(K,e_K,e_*^K)$ is the finite \'{e}tale  group scheme  $\underline{\Gamma}_S$ representing the functor
$$
\underline{\mathrm{Aut}}(K,e_K,e^K_*)(T\rightarrow S):= \mathrm{Aut}(K\times_S T, e_K, e^*_K).
$$
\end{definition}

The character $\lambda$ of Theorem \ref{theorem:thetaCharacter}, composed with $[i]: \underline{\ZZ/4\ZZ}_S \simeq \mu_4$,  gives by descent a group scheme homomorphism
\begin{equation}
\label{equation:schemeCharacter}
\lambda: \underline{\Gamma}_S \longrightarrow \mu_4,
\end{equation}
which is compatible under taking direct sums, as in Proposition \ref{proposition:directSums}.

\subsection{} Given a symplectic 4-group scheme with theta characteristic $(K,e_K,e^K_*)$, consider the constant triple $(\underline{\ZZ/4\ZZ}_S^{2g},e_4,e^{\pm}_*)$ of rank $2g$, equipped with the standard symplectic form $e_4$ and a theta characteristic $e^{\pm}_*$ of the same type (i.e. even or odd) as that of $e_*^K$. The functor on $S$-schemes given by  
$$
\{T\rightarrow S\} \mapsto \mathrm{Isom}\left((K\times_S T,e_K,e^K_*),(\underline{\ZZ/4\ZZ}_S^{2g},e_4,e^{\pm}_*)\right)
$$
is representable by a $\underline{\Gamma}_S$-torsor $\underline{\mathrm{Isom}}_S\left((K,e_K,e^K_*),(\underline{\ZZ/4\ZZ}_S^{2g},e_4,e^{\pm}_*)\right)$. Define
$$
\cM(K,e_K,e_*) := \lambda_*\,\underline{\mathrm{Isom}}\left((K,e_K,e^K_*),(\underline{\ZZ/4\ZZ}_S^{2g},e_4,e^{\pm}_*)\right),
$$
a $\mu_4$-torsor over $S$, whose formation is compatible under base-change.

\begin{definition}
\label{definition:thetaMultiplierBundle}
Given a symplectic 4-group scheme $(K\rightarrow S,e_K,e^{K}_*)$ with theta characteristic, the $\mu_4$-torsor $\cM(K,e_K,e^K_*)$ over $S$ is called the {\em theta multiplier bundle} associated to $(K,e_K,e_*)$. 
\end{definition}

By Proposition \ref{proposition:directSums}, the formation of theta multiplier bundles is compatible with direct sums. More precisely, given any two symplectic 4-group schemes $(K,e_K,e_*^K),(K',e_{K'},e_*^{K'})$ with theta characteristic over $S$, there is a $\mu_4$-torsor isomorphism
\begin{equation}
\label{equation:torsorDirectSums}
\cM(K\oplus K', e_K\oplus e_{K'}, e_*^K\oplus e_*^{K'}) \stackrel{\simeq}\rightarrow \cM(K,e_K,e^K_*)\otimes \cM(K',e_{K'},e_*^{K'}).
\end{equation}

\section{Determinant line bundles on abelian schemes}

\subsection{} Let now $\pi:A\rightarrow S$ be an abelian scheme of dimension $g$ with identity section $e:S\rightarrow A$. Let $\cL$ be an invertible $\cO_A$-module which has been normalized at the identity, i.e. we have chosen an $\cO_S$-module isomorphism
$
e^*\cL \simeq \cO_S. 
$
Any such invertible $\cO_A$-module defines a canonical morphism
$
\varphi_{\cL}: A \longrightarrow A^{t}
$
to the dual abelian scheme. Let
$$
K(\cL):= \ker \phi_{\cL},
$$
a commutative group scheme over $S$. When $\cL$ is relatively ample, $K(\cL)$ is finite flat over $S$ and it is canonically endowed with a non-degenerate symplectic pairing
$$
e_{\cL}: K(\cL)\times K(\cL) \longrightarrow \GG_m. 
$$
In this case, $\mathrm{rk}(K(\cL)) = d^2$, where $d$ is the degree of $\cL$, so that  $K(\cL)$ is \'{e}tale over $S[1/d]$ (\cite{Mumford:EqAbv1} \S 1, \cite{Mumford:EqAbv2} \S 6). 

Suppose next that $S$ is a scheme where $1/2\in \cO_S$ and suppose that the (relatively ample, normalized) invertible sheaf $\cL$ is {\em symmetric}. This means that there is a (unique) isomorphism of normalized invertible $\cO_A$-modules
$
\iota_{\cL}: [-1]^*\cL \stackrel{\simeq}\longrightarrow \cL,
$
where $[-1]:A \rightarrow A$ is the inversion morphism. In this situation the isomorphism $\iota_{\cL}$, restricted to the fixed locus $A[2]$ of $[-1]^*$, is multiplication by $\pm 1$ and thus it defines a function
$$
e^{\cL}_*: A[2] \longrightarrow \mu_2.
$$
This function is quadratic for the symmetric pairing $e_{\cL^2}$ (\cite{Mumford:EqAbv1},\S 2, Cor. 1). 

\begin{proposition}
\label{proposition:ThetaQuadraticStructure}
Let $\Theta$ be a normalized, relatively ample, symmetric invertible sheaf of degree $d=1$ over an abelian scheme $A\rightarrow S$ of dimension $g$. Then $(K(\Theta^4), e_{\Theta^4},e_*^{\Theta})$ is a symplectic 4-group of rank $2g$ with theta characteristic.
\end{proposition}

\begin{proof}
The invertible sheaf $\Theta^2$ (resp. $\Theta^4$) is also relatively ample and of degree $2^{g}$ (resp. $4^{g}$), so that $K(\Theta^2)$ (resp. $K(\Theta^4)$) is finite \'{e}tale of rank $2^{2g}$ (resp. $4^{2g}$). In fact, since $\varphi_{\cL}$ is additive in $\cL$, there is a canonical isomorphism 
$$
K(\Theta^2) = \ker 2\,\varphi_L \simeq A[2]
$$ 
and similarly $K(\Theta^4) = \ker 4\,\varphi_L \simeq A[4]$. In particular, the function $e_*^{\Theta}$ is a theta characteristic on the symplectic 4-group scheme $(K(\Theta^4), e_{\Theta^4})$. 
\end{proof}

Suppose additionally that $S$ is over $\mathbb{Z}[1/2,i]$. Given $\Theta$ as in Proposition \ref{proposition:ThetaQuadraticStructure}, we can form the corresponding theta multiplier bundle (Definition \ref{definition:thetaMultiplierBundle})
\begin{equation}
\label{equation:thetaMultiplierBundleDefinition}
\cM(\Theta):= \cM(K(\Theta^4), e_{\Theta^4},e_*^{\Theta}),
\end{equation}
which is a $\mu_4$-torsor over $S$. 

\begin{proposition}
\label{proposition:base-change}
The formation of $\cM(\Theta)$ is compatible under base-change; that is, given any two pairs $(A,\Theta)$ and $(A',\Theta')$ as in Proposition \ref{proposition:ThetaQuadraticStructure}, a base-change morphism of abelian schemes
$$
\begin{tikzcd}
A'\stackrel{\phi}\simeq  A\times_S S' \arrow{r} \arrow{d}  & A\arrow{d}\\
S' \arrow{r}{\varphi} & S,
\end{tikzcd}
$$
together with an $\cO_{S'}$-module isomorphism $\psi: \phi^*\Theta \simeq \Theta'$, then there is a canonical isomorphism
$
\varphi^*\cM(\Theta) \simeq \cM(\Theta')
$
as $\mu_4$-torsors. 
\end{proposition}

\begin{proof}
This is clear, since the formation of the symplectic 4-group $(K(\Theta^4), e_{\Theta^4})$ and the theta characteristic $e_*^{\Theta}$ are all compatible under such base-change. 
\end{proof}

In addition, the formation of $\mathcal{M}(\Theta)$ is compatible under taking direct products, i.e. for any two pairs $(A_1,\Theta_1),(A_2,\Theta_2)$ there is a canonical $\mu_4$-torsor isomorphism 
\begin{equation}
\label{equation:externalDirectSums}
\mathcal{M}(\Theta_1\boxtimes\Theta_2) \simeq \mathcal{M}(\Theta_1)\otimes\mathcal{M}(\Theta_2),
\end{equation}
as follows from \eqref{equation:torsorDirectSums}.

\subsection{} There is another $\mu_4$-torsor that can be canonically attached to the pair $(A,\Theta)$ of Proposition \ref{proposition:ThetaQuadraticStructure}. Namely, let
\begin{equation}
\label{equation:determinantDefinition}
\Delta(\Theta):= \pi_*\Theta^{\otimes 2}\otimes \underline{\omega}_{A/S}
\end{equation}
be the {\em determinant bundle}, where $\underline{\omega}_{A/S}:= \det(\pi_*\Omega^1_{A/S})$ is the Hodge bundle of $\pi:A\rightarrow S$. Since $\Theta$ is relatively ample of degree 1, the determinant bundle is an invertible sheaf over $S$ (\cite{MoretBailly:PinceauxAbv}, VIII.1.0), and its formation is compatible with base-change (\cite{MoretBailly:PinceauxAbv}, VIII.1.1.1). Moreover, there is an $\cO_S$-module isomorphism (\cite{FaltingsChai}, Theorem I.5.1) 
\begin{equation}
\label{equation:determinantBundleIs4Torsion}
\Delta(\Theta)^{\otimes 4} \simeq \cO_S
\end{equation}
which is compatible under base-change (\cite{Polishchuk-Determinants}, Remark after Theorem 0.2) giving $\Delta(\Theta)$ a canonical $\mu_4$-torsor structure over $S$. 

\begin{remark}
The right level of generality in which to study determinant bundles is that of a non-degenerate symmetric line bundle $\mathcal{L}$ over $A$ of arbitrary degree $d\geq 0$, in which case $\Delta(\mathcal{L}) := (\det R\pi_*\Theta)^{\otimes 2}\otimes \underline{\omega}_{A/S}$ is a `true' determinant (\cite{MoretBailly:PinceauxAbv}, \cite{FaltingsChai}, \cite{Polishchuk-Determinants}). 
\end{remark}

By the K\"{u}nneth formula, determinant bundles are also compatible under direct products, i.e. for any two pairs $(A_1,\Theta_1),(A_2,\Theta_2)$ there is a canonical $\mu_4$-torsor isomorphism 
\begin{equation}
\label{equation:determinantCompatibility}
\Delta(\Theta_1\boxtimes\Theta_2) \simeq \Delta(\Theta_1)\otimes\Delta(\Theta_2).
\end{equation}
We now `compute' determinants in the genus 1 case, following \cite{Polishchuk-Determinants}, \S 5.1.

\subsection{Case $g=1$, $e_*$ even}
\label{section:determinantsGenus1Even} Let $\pi: E\rightarrow S$ be an elliptic curve over a scheme $S$ with $1/2\in \cO_S$, and let $e:S\rightarrow E$ be its identity section. Let $P:S\rightarrow E$ be a non-trivial section of order 2. Then we claim that 
$
\Theta := \cO_E(P)
$
is a normalized, relatively ample, symmetric invertible sheaf of degree 1 over $E$ with even $e_*^{\Theta}$. Indeed $e^*\Theta\simeq\cO_S$ (since $P\neq e$), $\cO_E(P)$ is ample of degree 1 on geometric fibers and $\cO_E(P)\simeq \cO_E([-P])$ since $P$ is 2-torsion. The theta characteristic is even, since $e_*^{\Theta}$  is given by the formula (\cite{Mumford:EqAbv1}, \S2, Proposition 2)
$$
e_*^{\Theta}(x) = (-1)^{\mathrm{mult}_x(P) - \mathrm{mult}_e(P)}
$$
where $x\in E[2]$. We claim that there is an isomorphism
\begin{equation}
\label{equation:determinantGenus1Even}
\Delta(\Theta) \simeq \underline{\omega}_{E/S},
\end{equation}
induced by the isomorphism $\pi_*\cO_E(P) \simeq \cO_S$. The latter can be deduced from the adjunction exact sequence 
$$
0 \longrightarrow \cO_E \longrightarrow \cO_E(P) \longrightarrow P_*P^*\cO_E(P)\longrightarrow 0.
$$
Indeed, taking right-derived functors of $\pi_*$ we get a long exact sequence
$$
0\rightarrow \pi_*\cO_E \rightarrow \pi_*\cO_E(P) \rightarrow P^*\cO_E(P) \rightarrow R^1\pi_*\cO_E \rightarrow R^1\pi_*\cO_E(P) \rightarrow \ldots
$$
But $R^1\pi_*\cO_E(P) = 0$, as can be checked on the geometric fibers, and all the other sheaves in the exact sequence are locally free of rank 1, thus $\pi_*\cO_E(P) \simeq \pi_*\cO_E \simeq \cO_S$.

\subsection{Case $g=1$, $e_*$ odd}
\label{section:determinantsGenus1Odd}  In this case, let
$$
\Theta:= \cO_E(e)\otimes\Omega^1_{E/S}.
$$
Note that $e^*\cO_E(e) \simeq R^1\pi_*\cO_E \simeq \underline{\omega}_{E/S}^{-1}$, as can be deduced as above by taking right-derived functors of $\pi_*$ applied to the adjunction exact sequence
$$
0 \longrightarrow \cO_E \longrightarrow \cO_E(e) \longrightarrow e_*e^*\cO_E(e)\longrightarrow 0.
$$
Therefore $\Theta$ is normalized. It is also symmetric, relatively ample of degree 1 as before. The characteristic is now odd, and there is an isomorphism 
\begin{equation}
\label{equation:determinantGenus1Odd}
\Delta(\Theta) \simeq \underline{\omega}^{\otimes 3}_{E/S},
\end{equation}
as follows by noting that $\pi_*\cO_E(e)\simeq\cO_E$, as before, and by the projection formula applied to $\Omega^1_{E/S} = \pi^*\underline{\omega}_{E/S}$.

\section{The canonical key formula}
\label{section:canonicalKeyFormula}

Let $(A,\Theta)$ be a pair of an abelian scheme $A\rightarrow S$ of dimension $g$, where $S$ is a scheme over $R=\ZZ[1/2,i]$, together with a normalized, relatively ample, symmetric invertible sheaf $\Theta$ of degree $d=1$.  We have two canonical $\mu_4$-torsors over $S$ attached to this pair:
\begin{itemize}
\item[(i)]the {\em theta multiplier bundle}
$$
\cM(\Theta):= \cM(K(\Theta^4), e_{\Theta^4},e_*^{\Theta})
$$
of \eqref{equation:thetaMultiplierBundleDefinition}, and  
\item[(ii)] the {\em determinant bundle} 
$$
\Delta(\Theta):= \pi_*\Theta^{\otimes 2}\otimes \underline{\omega}_{A/S}
$$
of \eqref{equation:determinantDefinition}.
\end{itemize}

This section is devoted to showing the following `canonical key formula' (cp. the weaker $\mathrm{FCC}^{\mathrm{ab}}(\Spec(R),g,1)$ of \cite{MoretBailly:PinceauxAbv}, VIII.1.2), comparing the image of $\cM(\Theta)$ and $\Delta(\Theta)$ under the natural map
$$
\{\text{iso. classes of } \mu_4\text{-torsors over } S\}= H^1_{\mathrm{\acute{e}t}}(S,\mu_4)\longrightarrow H^1_{\mathrm{\acute{e}t}}(S,\GG_m)[4] = \Pic(S)[4].
$$

\begin{theorem}[Canonical Key Formula]
\label{theorem:FCC}
For each pair $(A\rightarrow S,\Theta)$ as above, there is a functorial (i.e. compatible under base-change) $\cO_S$-module isomorphism
$$
\Delta(\Theta)\simeq \cM(\Theta)^{-1}
$$
between the underlying invertible sheaves. 
\end{theorem}

\begin{proof}

Consider the algebraic stack $\widetilde{\sA}_g$ over $R$ classifying all pairs $(A,\Theta)$ as above (morphisms as in Proposition \ref{proposition:base-change}), and let $\Delta_g, \cM_g$ be the $\mu_4$-torsors defined over $\widetilde{\sA}_{g}$ by the functors
$
(A,\Theta) \mapsto \Delta(\Theta) $ and 
$(A,\Theta) \mapsto \mathcal{M}(\Theta)$, respectively. To prove Theorem \ref{theorem:FCC}, it suffices to show that there exists an isomorphism $\Delta_g\simeq\cM_g^{-1}$ as invertible sheaves over $\widetilde{\sA}_{g} $. Now the stack $\widetilde{\sA}_{g} $  is smooth (\cite{MoretBailly:PinceauxAbv}, VIII.3.2.4) with two irreducible components $\sA^{\pm }_{g}$, each classifying pairs $(A,\Theta)$ with $e_*^{\Theta}$ even (+) or odd (-) (\cite{MB:FEqn}, Prop. 1.1.4). For each $g_1,g_2 \geq 1$, there is a map
\begin{align*}
m_{g_1,g_2}: \widetilde{\sA}_{g_1}\times \widetilde{\sA}_{g_2} &\longrightarrow \widetilde{\sA}_{g_1+g_2} \\
((A_1,\Theta_1), (A_2,\Theta_2)) &\longmapsto (A_1\times A_2, \Theta_1\boxtimes\Theta_2)
\end{align*}
and $\mu_4$-torsor isomorphisms (\eqref{equation:determinantCompatibility} and \eqref{equation:externalDirectSums})
\begin{align*}
m_{g_1,g_2}^*\Delta_{g_1+g_2} &\simeq \Delta_{g_1}\otimes\Delta_{g_2} \\
m_{g_1,g_2}^*\cM_{g_1+g_2} &\simeq \cM_{g_1}\otimes\cM_{g_2}.
\end{align*}
Moreover, the map $m_{g_1,g_2}$ sends $\sA^{-}_{g_1}\times \sA^{- }_{g_2}$ to $\sA^{+}_{g_1+g_2}$,  $\sA^{+}_{g_1}\times \sA^{+ }_{g_2}$ to $\sA^{+}_{g_1+g_2}$ and so on (\cite{MB:FEqn},\S 1.3). It suffices then to prove $\Delta_g\simeq\cM_g^{-1}$ over $\sA^{-}_1$ and over $\sA^{+}_g$, $g\geq 3$. 

First, consider $\sA^{-}_1 $. Following a technique of Mumford (\cite{Mumford:Picard}, \S6), we are going to construct an explicit isomorphism $$\alpha_1^{-}:\Pic(\sA^{-}_1 )[4] \stackrel{\simeq}\longrightarrow \ZZ/4\ZZ,$$ and then show that the images of  $\Delta_1, \cM_1^{-1}$ under $\alpha_1^{-}$ are equal. To do so, let $k$ be an algebraically closed field of characteristic not 2, and consider the elliptic curve $E/k$ given by the Weierstrass equation 
$
y^2 = x^3 - x.
$
Let $\Theta^{-}_1 := \cO_{E}(e)\otimes\Omega^1_{E/k}$. The pair $(E,\Theta^{-}_1)$ is classified by a point (Section \ref{section:determinantsGenus1Odd})
$$
\kappa_1^-: \Spec(k) \longrightarrow \sA^{-}_1 .
$$
The curve $E$ has a special automorphism of order 4 defined by
\begin{equation}
\label{equation:zeta4Auto}
[i](x) = -x, \quad [i](y) = i\,y,
\end{equation}
where $i \in k^{\times}$ is a choice of primitive 4-th root of unity. Under a suitable choice of basis for $E[4] \simeq (\ZZ/4\ZZ)^2$, we may assume that the action of $[i]$ is given by the matrix $\smalltwobytwo{0}{-1}{1}{0}$ and that $e_*^{\Theta_1^-}(a,b) = (-1)^{a+b+ab}$. Thus $[i]$ extends to an automorphism of the pair $(E,\Theta^-_1)$, since it preserves $e_*$. Now given an invertible sheaf $\cL$ on $\sA^{-}_1 $, we may uniquely define a 4-th root of unity $\alpha^-_1(\cL)$ by 
$$
\alpha^-_1(\cL):= [i]^* \in \Aut(\kappa_1^{-\,*}{\cL}) \simeq k^{\times},
$$
which gives a homomorphism
$
\alpha^-_1: \Pic(\sA^{-}_1 )\longrightarrow \ZZ/4\ZZ.
$
If we let $\cL = \Delta_1$, then $\kappa_1^{-\,*}\Delta_1 \simeq \underline{\omega}_{E/k}^{\otimes 3}$ by \eqref{equation:determinantGenus1Odd}. Now $\underline{\omega}_{E/k}$ is the 1-dimensional $k$-vector space generated by the regular differential $\frac{dx}{y}$ and thus
$$
\alpha^-_1(\Delta_1) = i^3 = i^{-1},
$$
as follows from the explicit formulas \eqref{equation:zeta4Auto}  defining $[i]$. Since $i$ is primitive, $\alpha^-_1$ is surjective. Moreover, since there is only one odd, normalized, symmetric, relatively ample invertible sheaf of degree 1 on an elliptic curve, we have $\sA^{-}_1  = \sM_{1,1}[1/2,i]$, where $\sM_{1,1}$ is the moduli stack of elliptic curves, for which we have (e.g. \cite{FultonOlsson})
$$
\Pic(\sM_{1,1}[1/2,i])[4] \simeq \ZZ/4\ZZ.
$$
Therefore $\alpha^-_1$ is an isomorphism when restricted to $\Pic(\sA^{-}_1 )[4]$. To show that $\Delta_1\simeq \cM_1^{-1}$, note that 
$$
\alpha^-_1(\cM_1^{-1}) = \lambda^{-1}\smalltwobytwo{0}{-1}{1}{0} = i^{-1},
$$
as follows from the computations of Section \ref{section:lambdaGenus1Odd}.

Next, we employ the same technique for the case $g\geq 3$ even. In particular, we are going to construct an explicit isomorphism   
$$
\alpha^+_g: \Pic(\sA^{+}_g)[4]\stackrel{\simeq}\longrightarrow \ZZ/4\ZZ
$$
for all $g\geq 3$, and then prove that $\alpha^+_g(\Delta_g) = \alpha^+_g(\cM_g^{-1})$. To do so, consider again the elliptic curve $E/k$. Let $\Theta^+_1 := \cO_{E}(P)$, where $P$ is the unique non-trivial point of order 2 fixed by $\Aut(E/k)$. The pair $(E,\Theta^+_1)$ is classified by a point (Section \ref{section:determinantsGenus1Even})
$
\kappa^+_1: \Spec(k) \rightarrow \sA^{+}_1.
$
If $g$ is odd, consider the point
$$
\kappa^{+}_g: \Spec(k) \longrightarrow \sA^{+}_g
$$
classifying $(E^{\times g},(\Theta^+_1)^{\boxtimes g})$. The special automorphism $[i]$ preserves the characteristic and it extends to an automorphism $[i]^{\times g}$ of the pair $(E^{\times g},(\Theta^+_1)^{\boxtimes g})$. Given an invertible sheaf $\cL$ on $\sA^{+}_g$, the action of $[i]^{\times g}$ on $\kappa^{+\,*}_g(\cL)$ gives a homomorphism  $\alpha^+_g: \Pic(\sA^{+}_g)\longrightarrow \ZZ/4\ZZ$. This is surjective, since \eqref{equation:determinantGenus1Even} shows that
$$
\kappa^{+\,*}_g(\Delta_g) \simeq \underline{\omega}_{E/k}^{g} = k\left(\frac{dx}{y}\right)^{\otimes g}
$$
and therefore $\alpha^+_g(\Delta_g) = i^g$, a primitive 4-th root of unity. If $g$ is even, apply the same argument with the automorphism $[i]^{\times g - 1}\times\id$ replacing $[i]^{\times g}$. In both cases, there is an isomorphism (\cite{Polishchuk-Determinants}, Theorem 5.6)
$$
\Pic(\sA^{+}_g)[4] \simeq \ZZ/4\ZZ
$$ 
thus $\alpha^+_g$ restricts to the desired isomorphism on 4-torsion.

To prove that $\Delta_g \simeq \cM_g^{-1}$, note that for $g$ odd
$$
\alpha_g^+(\cM_g^{-1}) = \lambda^{-g}\smalltwobytwo{0}{-1}{1}{0} = i^g,
$$ 
as follows from the computations of Section \ref{section:determinantsGenus1Even} and by the compatibility of $\lambda$ under direct sums. Similarly, for $g$ even we have
$$
\alpha_g^+(\cM_g^{-1}) = \lambda^{-g+1}\smalltwobytwo{0}{-1}{1}{0} = i^{g-1},
$$
which shows that $\Delta_g \simeq \cM_g^{-1}$ in all cases. 
\end{proof}

\begin{remark}
\label{remark:theta-function-free}
Theorem \ref{theorem:FCC} can also be proved using the classical transformation laws of theta functions. The advantage of using the `theta-function-free' approach above is that we can now use Theorem \ref{theorem:FCC} to give a new proof of the transformation laws of theta functions (this is done in Section \ref{section:classicalFormulas} below).
\end{remark}

\begin{remark}
Any two choices of isomorphism $\Delta_g \simeq \cM_g^{-1}$ must differ by an element of $\Gamma(\sA^{\pm }_g,\GG_m)$. The same arguments as in \cite{MB:FEqn}, \S 1 show that $\Gamma(\sA^{\pm}_g,\GG_m) = R^{\times}$, thus the isomorphism of Theorem \ref{theorem:FCC} is unique up to multiplication by a constant in $R^{\times}$. By going through the above arguments more carefully, it seems possible to reduce the ambiguity to $\mathbb{Z}^{\times} = \{\pm 1\}$. 
\end{remark}

\section{Algebraic and analytic functional equations}
\label{section:functionalEquation}

We now explain how the canonical key formula (Theorem \ref{theorem:FCC}, or rather its refinement \eqref{equation:keyFormulaSquareRoot} below) can be viewed as the algebraic analog of the functional equation of Riemann's theta function. For simplicity, we only treat the case of even theta characteristic, the odd case differing only in the explicit analytic formulas for the theta function. Thus let $(\pi:\sA\rightarrow \sA^{+}_g ,\Theta)$ be the universal pair over $\sA^{+}_g\rightarrow \Spec(R)$. The canonical key formula gives by duality an isomorphism
 $$
 \left(\pi_*\Theta\right)^{-2} \simeq \cM(\Theta)\otimes\underline{\omega},
 $$
 where $\cM(\Theta) = \cM_g$ and $\underline{\omega}$ is the Hodge bundle of $\pi:\sA\rightarrow \sA^{+}_g$. The {\em bundle of half-forms} 
 $$
 \underline{\omega}^{1/2}_{\Theta} := \sqrt{\cM(\Theta)\otimes\underline{\omega}},
 $$
 a square-root which a priori can only be defined over a $\mu_2$-gerbe over $\sA^{+}_g$, descends to a well-defined invertible sheaf over $\sA^{+}_g$, together with an isomorphism 
  \begin{equation}
  \label{equation:keyFormulaSquareRoot}
 \iota_{\mathrm{alg}}: \left(\pi_*\Theta\right)^{-1} \stackrel{\simeq}\rightarrow  \underline{\omega}^{1/2}_{\Theta}.
 \end{equation}
The algebraic Riemann theta function is a section of  $\left(\pi_*\Theta\right)^{-1}$, defined as follows (e.g. \cite{FaltingsChai}, \S I.5): since $\Theta$ is normalized, there is a well-defined `evaluation-at-the-identity' map 
$
\mathrm{ev}: e^*\Theta \longrightarrow \cO_{\sA^{+}_g}
$
which by adjunction gives a section $\vartheta_g$ of $(\pi_*\Theta)^{-1} = \Hom(\pi_*\Theta, \cO_{\sA^{+}_g})$. By \eqref{equation:keyFormulaSquareRoot}, we know that $\vartheta_g$ maps isomorphically to a section of $ \underline{\omega}^{1/2}_{\Theta}$, an `algebraic modular form' of weight 1/2 over $\sA^{+}_g$.

\subsection{}
\label{section:classicalFormulas} 

Over the category of analytic spaces, isomorphism classes of line bundles over the analytic quotient stack $\left[\Gamma^{+}_g(1,2)\backslash \mathfrak{h}_g\right] = \sA^{+}_{g,\rm{an}}$ are in 1-1 correspondence with group cohomology classes in $H^1(\Gamma^{+}_g(1,2), \cO^{\times}_{\mathfrak{h}_g})$, since $\mathfrak{h}_g$ is a Stein manifold. In particular, let $\mathrm{pr}: \mathfrak{h}_g\rightarrow  \sA^{+}_{g,\rm{an}}$ be the projection map, classifying the universal abelian variety $$
A_g:=\CC^g\times \mathfrak{h}_g/\Lambda_g,\quad \Lambda_g = \{(z,\tau)\in \CC^g\times\mathfrak{h}_g : z\in \ZZ^g + \tau\ZZ^g\},$$ together with the symmetric line bundle $Th$ over $A_g\rightarrow \mathfrak{h}_g$ defined by the divisor of the (2-variable) Riemann theta function $\sum_{n\in \ZZ^g} e^{\pi i\, n^t\tau n + 2\pi i n^tz } 
$. The line bundle $\mathrm{pr}^* \underline{\omega}^{1/2}_{\Theta}$ is trivialized over $\mathfrak{h}_g$ by the section 
$$
s_g := \sqrt{c\,dz_1\cdots dz_g},
$$
where $c$ is a constant section of $\cM(Th)$, the theta multiplier bundle over the universal abelian variety $A_g$, and $z=(z_1, \ldots, z_g)$ are coordinates on $\CC^g$. This choice of trivialization corresponds to the 1-cocycle  $j_{1/2}\in Z(\Gamma_g(1,2)^+,\cO_{\mathfrak{h}_g^{\times}})$ given by
$$
j_{1/2}: \gamma = \smalltwobytwo{A}{B}{C}{D}  \longmapsto \sqrt{\lambda(\gamma)\cdot \det(C\tau + D)},
$$
where $\lambda: \Gamma^{+}_g(1,2) \rightarrow \mu_4$ is the character obtained from the $\lambda$ of Section \ref{section:thetaMultipliers} by factoring through the principal congruence subgroup $\Gamma_g(4)$ (technically, this is a 1-cocycle of the metaplectic cover $\widetilde{\Gamma}^{+}_g(1,2)$, but formula \eqref{equation:keyFormulaSquareRoot} shows that the square-roots can be chosen in a compatible way so that the resulting cocycle descends to $\Gamma^{+}_g(1,2)$). On the other hand, the line bundle $\pr^*(\pi_*\Theta)^{-1}$ is trivialized over $\mathfrak{h}_g$ by Riemann's theta function $\vartheta_g(\tau) = \sum_{n\in \ZZ^g} e^{\pi i\, n^t\tau n} $, giving rise to a 1-cocycle
$$
\gamma \longmapsto \vartheta_g(\gamma\tau)\vartheta_g(\tau)^{-1}.
$$
By the classical functional equation of Riemann's theta function, the composition of the two trivializations $s_g$ and $\vartheta_g$ is a $\Gamma^+(1,2)$-invariant isomorphism
$$
 \iota_{\mathrm{an}}: \mathrm{pr}^*(\pi_*\Theta)^{-1}\stackrel{\simeq}\longrightarrow \mathrm{pr}^* \underline{\omega}^{1/2}_{\Theta},
$$
and thus it must descend to an isomorphism over the quotient $\sA^{+}_{g,\rm{an}}$. Another such isomorphism is given by the base-change $\iota_{\rm{alg},\CC}$ of \eqref{equation:keyFormulaSquareRoot} from $R = \ZZ[1/2,i]$ to $\CC$. The composition $\iota_{\rm{alg},\CC}\circ\iota_{\rm{an}}^{-1}$ is then a non-vanishing analytic function on $\sA^{+}_{g,\rm{an}}$ (compatible under products of abelian varieties) and is thus constant (\cite{MB:FEqn}, 1.4.1). In fact, the arguments of \cite{MB:FEqn} can be employed to show that $\iota_{\rm{alg},\CC}= \pm \iota_{\rm{an}}$. In this sense, the isomorphism $\iota_{\rm{alg}}$ of \eqref{equation:keyFormulaSquareRoot} can be viewed as the algebraic analog to the functional equation of Riemann's theta function, since the existence of the functional equation is equivalent to the existence of the $\Gamma_g^+(1,2)$-equivariant isomorphism $\iota_{\rm{an}}$.

\end{document}